\documentclass[12pt]{amsart}
\usepackage{amssymb}
\usepackage[all]{xy}
\usepackage{mathrsfs}
\newtheorem{theorem}{Theorem}[]

\newtheorem{definition}[theorem]{Definition}

\newtheorem{problem}[theorem]{Problem}
\def\to{\rightarrow}
\def\f{\mathfrak}
\def\c{\mathcal}
\def\b{\mathbf}
\def\r{\mathrm}

\def\sl{\langle}
\def\sr{\rangle}
\def\ot{\otimes}
\begin{document}
\baselineskip15pt
\title[quantum metrics on noncommutative spaces]{Quantum metrics on noncommutative spaces}
\author[M. M. Sadr]{Maysam Maysami Sadr}
\address{Department of Mathematics\\
Institute for Advanced Studies in Basic Sciences\\
P.O. Box 45195-1159, Zanjan 45137-66731, Iran}
\email{sadr@iasbs.ac.ir}
\begin{abstract}
We introduce two new formulations for the notion of ``quantum metric on noncommutative space''. For a compact noncommutative
space associated to a unital C*-algebra, our quantum metrics are elements of the spatial tensor product of the C*-algebra with itself.
We consider some basic properties of these new objects, and state some connections with the Rieffel theory of compact quantum metric spaces.
The main gap in our work is the lack of a nonclassical example even in the case of C*-algebras of matrices.
\end{abstract}
\maketitle
There are at least two mathematically rigorous formulations for the notion of ``quantum (noncommutative) metric space'' in the literatures.
The famous one is due to Rieffel, and the other has been recently introduced by G. Kuperberg and N. Weaver.
Following some ideas from Connes \cite[Chapter VI]{Connes1} in noncommutative Riemannian
geometry \cite{Connes2}, Rieffel has introduced the notions of ``compact quantum metric space'' and
``quantum Hausdorff-Gromov distance'' \cite{Rieffel2},\cite{Rieffel3},\cite{Rieffel1}.
In his theory, a compact quantum metric space $\f{q}\c{A}$ is identified with the state space of a unital C*-algebra $\c{A}$
(or more generally, with the state space of an order unit space) together with a weak*-compatible metric which must be induced by a
``Lipschitz seminorm'' on $\c{A}$ via Monge-Kantorovich's formula. Thus in the Rieffel theory the role of quantum metrics
is played by Lipschitz seminorms. In Kuperberg-Weaver theory \cite{KuperbergWeaver1} the noncommutative space is distinguished by a
von Neumann algebra $\c{M}\subseteq\b{L}(\c{H})$ and the role of quantum metric is played by a specific one-parameter family $\{\c{V}_t\}_{t\geq0}$
of weak closed operator systems in $\b{L}(\c{H})$ such that $\c{V}_0=\c{M}'$. This construction also can be characterized by a specific
``quantum distance function'' between projections of the von Neumann algebra $\c{M}\bar{\ot}\b{L}(\ell^2)$.
In this note, we introduce two new models for ``quantum metrics on noncommutative spaces''. Our formulations
are natural translations of the concept of ``(ordinary) metric'' into noncommutative geometric language.

For preliminaries on C* and von Neumann algebras we refer the reader to \cite{KadisonRingrose1} or \cite{Takesaki1}.
Let $\c{X}$ be a compact metrizable space. A function $\rho$ is a compatible metric on $\c{X}$  if and only if
$\rho\in\b{C}(\c{X}^2)=\b{C}(\c{X})\check{\ot}\b{C}(\c{X})$ and the following five conditions are satisfied for $x,y,z\in\c{X}$.
(i) $\rho(x,y)\geq0$. $(ii)$ $\rho(x,x)=0$. (iii) if $\rho(x,y)=0$ then $x=y$. (iv) $\rho(x,y)=\rho(y,x)$. (v) $\rho(x,y)\leq\rho(x,z)+\rho(z,y)$.
Let $\c{A}$ be a unital C*-algebra. Suppose that an element $\rho\in\c{A}\check{\ot}\c{A}$ deserves to be called a compatible
metric on $\f{q}\c{A}$. Then $\rho$ must be satisfied the following analogues of (i),(iv),(v).
\begin{enumerate}
\item[(i)$'$] $\rho\in(\c{A}\check{\ot}\c{A})^+$.
\item[(iv)$'$] $\f{F}(\rho)=\rho$, where $\f{F}:\c{A}\check{\ot}\c{A}\to\c{A}\check{\ot}\c{A}$ denotes the fillip.
\item[(v)$'$] $\f{M}(\rho)\leq \rho\ot1+1\ot\rho$, where $\f{M}:\c{A}\check{\ot}\c{A}\to\c{A}\check{\ot}\c{A}\check{\ot}\c{A}$
denotes the *-morphism that puts $1$ in the mid position, i.e. $\f{M}(a\ot b):=a\ot 1\ot b$ ($a,b\in\c{A}$).
\end{enumerate}
There are many ways to state the noncommutative analogues of (ii) and (iii). But it seems that the most effective and applicable
way is as follows. Let $\pi:\c{A}\to\b{L}(\c{H})$ denote a representation for $\c{A}$ by bounded operators on a Hilbert space $\c{H}$.
We say that $\pi$ is an \emph{atomic} representation if there is a family $\{\pi_i:\c{A}\to\b{L}(\c{H}_i)\}$ of
pairwise inequivalent irreducible representations of $\c{A}$ such that $\pi=\oplus \pi_i$. (Note that our atomic representations
are special cases of the atomic representations defined in \cite{Takesaki1}.)
Then it follows from \cite[Corollary 10.3.9]{KadisonRingrose1} that the enveloping von Neumann algebra $\pi(\c{A})''$ is equal to $\oplus_i\b{L}(\c{H}_i)$.
Let $\pi:\c{A}\to\b{L}(\c{H})$ be a faithful atomic representation of $\c{A}$. We consider $\c{A}$ as a subalgebra of $\b{L}(\c{H})$
and write $\c{A}''$ for $\pi(\c{A})''$. The \emph{characteristic function of the diagonal} (w.r.t. $\pi$) of $\f{q}\c{A}\times\f{q}\c{A}$
is denoted by $P_{\delta}$ and defined to be the supremum of the family of all projections of the form $p\ot p$ in $\c{A}''\bar{\ot}\c{A}''=(\c{A}\check{\ot}\c{A})''\subseteq\b{L}(\c{H}\bar{\ot}\c{H})$
such that $p\in\c{A}''$ is a minimal projection.
(In the classical case that $\c{A}=\b{C}(\c{X})$, if we choose $\pi$ to be the \emph{reduced atomic
representation} then $\pi(\b{C}(\c{X}))''$ is isomorphic to $\ell^\infty(\c{X})$ and $P_{\delta}$ is identified with the usual characteristic
function of the diagonal of $\c{X}\times\c{X}$.)
The analogues of (ii) and (iii) are as follows.
\begin{enumerate}
\item[(ii)$'$] $\rho P_{\delta}=P_{\delta}\rho=0$.
\item[(iii)$'$] Let $\c{H}_{\delta}$ denote the the image of the projection $P_{\delta}$ in $\c{H}\bar{\ot}\c{H}$.
Then, $0$ is not an eigenvalue of the operator $\rho|_{\c{H}_\delta^\perp}\in\b{L}(\c{H}_\delta^\perp)$.
\end{enumerate}
\begin{definition}\label{D1}
Let $\c{A}$ be a unital C*-algebra and let $\pi:\c{A}\to\b{L}(\c{H})$ be a faithful atomic representation.
A (compatible) quantum metric w.r.t. $\pi$ on $\f{q}\c{A}$ is an element $\rho\in\c{A}\check{\ot}\c{A}$ satisfying (i)$'$-(v)$'$.
In this case we call $(\c{A},\rho,\pi)$ a compact quantum metric space.
\end{definition}
Let $(\c{A},\rho,\pi)$ be a compact quantum metric space. Comparing with the classical case, it is natural that we consider the value $\|\rho\|$
as the diameter of $\rho$. It is clear that if $(\c{X},\rho)$ is an ordinary compact metric space then $(\b{C}(\c{X}),\rho,\pi)$
is a compact quantum metric space where $\pi$ is an arbitrary atomic representation of $\b{C}(\c{X})$.
(Indeed, it is easily checked that for any of such representation $\pi(\b{C}(\c{X}))''$ is isomorphic to $\ell^\infty(\c{X}_0)$
where $\c{X}_0$ is a dense subspace of $\c{X}$.)

We show that there is no quantum metric on the \emph{two-point noncommutative space} $\f{q}\b{M}_2$.
Let $\c{A}=\b{M}_2$ be the algebra of complex $2\times 2$ matrices. Then $\pi=\r{id}$ is an atomic representation of $\c{A}$ on $\c{H}=\mathbb{C}^2$.
Let $\{e_1,e_2\}$ (resp. $\{f_1,\cdots,f_4\}$) denote the Euclidean basis of $\mathbb{C}^2$ (resp. $\mathbb{C}^4$). We identify
$\mathbb{C}^2\ot\mathbb{C}^2$ with $\mathbb{C}^4$ via $e_1\ot e_1\mapsto f_1$, $e_2\ot e_2\mapsto f_4$, $e_1\ot e_2\mapsto f_2$,
$e_2\ot e_1\mapsto f_3$. Then $\b{M}_2\ot\b{M}_2$ is identified with $\b{M}_4$ via
$\sum_{i,j,k,\ell}\lambda_{ijk\ell}1_{ij}\ot1_{k\ell}\mapsto$
$$\left(\begin{array}{cccc}
  \lambda_{1111} & \lambda_{1112} & \lambda_{1211} & \lambda_{1212} \\
  \lambda_{1121} & \lambda_{1122} & \lambda_{1221} & \lambda_{1222} \\
  \lambda_{2111} & \lambda_{2112} & \lambda_{2211} & \lambda_{2212} \\
  \lambda_{2121} & \lambda_{2122} & \lambda_{2221} & \lambda_{2222} \\
\end{array}\right).$$
With these identifications, $P_\delta$ is the projection onto the linear subspace generated by $f_1,f_4,f_2+f_3$, and hence
$$P_\delta=\left(\begin{array}{cccc}
1 & 0 & 0 & 0 \\
0 & 1/2 & 1/2 & 0 \\
0 & 1/2 & 1/2 & 0 \\
0 & 0 & 0 & 1 \\
\end{array}\right).$$
Suppose that $\rho\in\b{M}_4$ satisfies (i)$'$,(ii)$'$,(iii)$'$,(iv)$'$. Then $\rho$ must be of the form
$$\rho=\left(\begin{array}{cccc}
0 & 0 & 0 & 0 \\
0 & \lambda & -\lambda & 0 \\
0 & -\lambda & \lambda & 0 \\
0 & 0 & 0 & 0\\
\end{array}\right),$$
for some real number $\lambda>0$. The $8\times 8$ matrix $M=\rho\ot1+1\ot\rho-\f{M}(\rho)$ is equal to
$$M=\lambda\left(\begin{array}{cccccccc}
  0 & 0 & 0 & 0 & 0 & 0 & 0 & 0 \\
  0 & 0 & -1 & 0 & 1 & 0 & 0 & 0 \\
  0 & -1 & 2 & 0 & -1 & 0 & 0 & 0 \\
  0 & 0 & 0 & 0 & 0 & -1 & 1 & 0 \\
  0 & 1 & -1 & 0 & 0 & 0 & 0 & 0 \\
  0 & 0 & 0 & -1 & 0 & 2 & -1 & 0 \\
  0 & 0 & 0 & 1 & 0 & -1 & 0 & 0 \\
  0 & 0 & 0 & 0 & 0 & 0 & 0 & 0
\end{array}\right).$$
For any vector $X=(x_1,\cdots,x_8)\in\mathbb{R}^8$ we have
$$\lambda^{-1}\sl MX,X\sr=(x_3-x_2-x_5)^2+(x_3^2-x_2^2-x_5^2)+(x_6-x_4-x_7)^2+(x_6^2-x_4^2-x_7^2).$$
Thus $M$ is not positive and hence $\rho$ dose not satisfy (v)$'$.

Although we just mentioned a negative result on the existence of quantum metrics but it seems that there must be a huge class of quantum
metrics on $\f{q}\b{M}_n$ for $n\geq3$. Indeed we invite and request the readers of this note, specially who are familiar with computational matrix-softwares,
to find some explicit examples of quantum metrics on $\f{q}\b{M}_n$.

Similar to the case of ordinary metric spaces we have the following three theorems.
\begin{theorem}\label{T1}
Let $\rho_1$ and $\rho_2$ be quantum metrics on $\f{q}\c{A}$ w.r.t. the same representation $\pi$ of $\c{A}$. Then, for every
positive real number $r$, $\rho_1+r\rho_2$ is a quantum metric on $\f{q}\c{A}$ w.r.t. $\pi$.
\end{theorem}
\begin{proof}
Straightforward.
\end{proof}
\begin{theorem}\label{T2}
Let $(\c{A}_1,\rho_1,\pi_1)$ and $(\c{A}_2,\rho_2,\pi_2)$ be compact quantum metric spaces
and let $r$ be a real number not less than $2^{-1}\max(\|\rho_1\|,\|\rho_2\|)$. Then
$(\c{A}_1\oplus\c{A}_2,\rho,\pi_1\oplus\pi_2)$ is a compact quantum metric space where
$\rho=\rho_1+\rho_2+r1_{\c{A}_1\check{\ot}\c{A}_2}+r1_{\c{A}_2\check{\ot}\c{A}_1}$.
\end{theorem}
\begin{proof}
The conditions (i)$'$ and (iv)$'$ are trivial for $\rho$. (ii)$'$ and (iii)$'$ follows from the fact that
any minimal projection in $(\pi_1\oplus\pi_2)(\c{A}_1\oplus\c{A}_2)''=\pi_1(\c{A}_1)''\oplus\pi_2(\c{A}_2)''$ is a minimal projection of
$\pi_1(\c{A}_1)''$ or of $\pi_2(\c{A}_2)''$. Let $\f{M}$,$\f{M}_1$,$\f{M}_2$ denote the corresponding morphisms as in (v)$'$ respectively for
$\c{A}_1\oplus\c{A}_2$,$\c{A}_1$,$\c{A}_2$. With the notation $1_{ijk}:=1_{\c{A}_i\check{\ot}\c{A}_j\check{\ot}\c{A}_k}$ we have,
$$\f{M}(\rho_1)\leq \f{M}_1(\rho_1)+2r1_{121},\hspace{5mm}
\f{M}(\rho_2)\leq \f{M}_2(\rho_2)+2r1_{212}.$$
It follows that
\begin{align*}
\f{M}(\rho)&=\f{M}(\rho_1)+\f{M}(\rho_2)+r1_{112}+r1_{122}+r1_{211}+r1_{221}\\
&\leq \f{M}_1(\rho_1)+\f{M}_2(\rho_2)+2r1_{121}+2r1_{212}+r1_{112}+r1_{122}+r1_{211}+r1_{221}\\
&\leq \rho_1\ot1_1+1_1\ot\rho_1+\rho_2\ot1_2+1_2\ot\rho_2\\
&\hspace{5mm}+2r1_{121}+2r1_{212}+r1_{112}+r1_{122}+r1_{211}+r1_{221}\\
&\leq \rho_1\ot1_1+\rho_1\ot1_2+\rho_2\ot1_1+\rho_2\ot1_2+r1_{121}+r1_{122}+r1_{211}+r1_{212}\\
&\hspace{5mm}+1_1\ot\rho_1+1_2\ot\rho_1+1_1\ot\rho_2+1_2\ot\rho_2+r1_{112}+r1_{212}+r1_{121}+r1_{221}\\
&=\rho\ot1+1\ot\rho.
\end{align*}
\end{proof}
\begin{theorem}\label{T3}
Let $(\c{A}_1,\rho_1,\pi_1)$ and $(\c{A}_2,\rho_2,\pi_2)$ be compact quantum metric spaces.
Then $(\c{A}_1\check{\ot}\c{A}_2,\rho,\pi_1\ot\pi_2)$ is a compact quantum metric space where
$$\rho=(\rho_1\ot1_{\c{A}_2\check{\ot}\c{A}_2}+1_{\c{A}_1\check{\ot}\c{A}_1}\ot\rho_2)
\in(\c{A}_1\check{\ot}\c{A}_1)\check{\ot}(\c{A}_2\check{\ot}\c{A}_2)\cong(\c{A}_1\check{\ot}\c{A}_2)\check{\ot}(\c{A}_1\check{\ot}\c{A}_2).$$
\end{theorem}
\begin{proof}
Straightforward.
\end{proof}
Let us now consider some relations between our model of ``compact quantum metric space'' and the other model introduced by Rieffel \cite{Rieffel1}.
Let $(\c{A},\rho,\pi)$ be a compact quantum metric space. We are able to define a new seminorm on $\c{A}$ which generalizes the Lipschitz
seminorm for continuous functions on an ordinary metric space. Let $\c{H}$ denote the Hilbert space of $\pi$ and let $\c{H}_\delta$
be as in (iii)$'$. Let $\rho^{-1}$ denote the inverse of the operator $\rho|_{\c{H}_\delta^\perp}\in\b{L}(\c{H}_\delta^\perp)$.
For any $a\in\c{A}$, the Lipschitz seminorm $\|a\|_{Lip}$ are defined to be the (possibly infinite) value $\|(a\ot1-1\ot a)\rho^{-1}\|$,
that is the operator norm of $(a\ot1-1\ot a)\rho^{-1}$ as an operator from the image of $\rho|_{\c{H}_\delta^\perp}$ into $\c{H}\bar{\ot}\c{H}$.
For $a,b\in\c{A}$ with $ab=ba$ the Leibnitz rule is satisfied:
\begin{align*}
\|ab\|_{Lip}&=\|(ab\ot1-1\ot ab)\rho^{-1}\|\\
&=\|(ab\ot1-a\ot b+a\ot b-1\ot ab)\rho^{-1}\|\\
&=\|[(a\ot1)(b\ot 1-1\ot b)+(a\ot1-1\ot a)(1\ot b)]\rho^{-1}\|\\
&\leq\|(a\ot1)(b\ot 1-1\ot b)\rho^{-1}\|+\|(1\ot b)(a\ot1-1\ot a)\rho^{-1}\|\\
&\leq\|a\|\|b\|_{Lip}+\|a\|_{Lip}\|b\|
\end{align*}
Also it is clear that for any normal element $a$ we have $\|a\|_{Lip}=\|a^*\|_{Lip}$. The seminorm $\|\cdot\|_{Lip}$ gives rise to a semimetric
on the state space $S(\c{A})$ of $\c{A}$ via Monge-Kantorovich formula:
$$d(\phi,\psi):=\sup_{a^*=a,\|a\|_{Lip}\leq1}|\sl\phi-\psi,a\sr|\hspace{5mm}(\phi,\psi\in S(\c{A})).$$
We give an upper bound for $d(\phi,\psi)$ in the case that $\phi$ and $\psi$ are some special pure states of $\c{A}$: Let $\pi$
be the direct sum of $\{\pi_i:\c{A}\to\b{L}(\c{H}_i)\}$. Suppose that $i\neq j$ and let $v$ and $w$ be two unit vectors respectively in $\c{H}_i$
and $\c{H}_j$. Let $\phi$ and $\psi$ be pure states on $\c{A}$ defined  respectively by $a\mapsto\sl\pi_i(a)v,v\sr$ and $a\mapsto\sl\pi_j(a)w,w\sr$.
Let $a$ be a self-adjoint element of $\c{A}$ with $\|a\|_{Lip}\leq1$. Since $v\ot w\in\c{H}_\delta^\perp$, we have
$\|a(v)\ot w-v\ot a(w)\|\leq\|\rho(v\ot w)\|$. Thus,
\begin{align*}
|\sl\phi-\psi,a\sr|^2&=\sl a(v),v\sr^2+\sl a(w),w\sr^2-2\sl a(v),v\sr\sl a(w),w\sr\\
&\leq\sl a(v),a(v)\sr+\sl a(w),a(w)\sr-2\sl a(v),v\sr\sl a(w),w\sr\\
&=\sl a(v),a(v)\sr\sl w,w\sr+\sl a(w),a(w)\sr\sl v,v\sr-2\sl a(v),v\sr\sl a(w),w\sr\\
&=\|a(v)\ot w-v\ot a(w)\|^2\leq\|\rho(v\ot w)\|^2.
\end{align*}
This shows that $d(\phi,\psi)\leq\|\rho(v\ot w)\|$.

As we saw above the most problematic part of the definition of a quantum metric is the translation of conditions (ii) and (iii).
We now translate these conditions in another way where there is no using of enveloping von Neumann algebras.
Let $\c{A}$ be a unital \emph{spatially continuous multiplication} C*-algebra, that means the multiplication of $\c{A}$, $m:a\ot b\mapsto ab$,
is continuous w.r.t. spatial tensor norm (e.g. $\c{A}$ is abelian or finite dimensional). For $\rho\in\c{A}\check{\ot}\c{A}$ satisfying
(i)$'$, consider the following conditions.
\begin{enumerate}
\item[(ii)$''$] $m(\rho)=0$.
\item[(iii)$''$] For every positive element $\nu\in\c{A}\check{\ot}\c{A}$ with $m(\nu)=1$ and $\f{F}(\nu)=\nu$,
the element $\rho+\nu$ is invertible in $\c{A}\check{\ot}\c{A}$.
\end{enumerate}
In the case that $\c{A}=\b{C}(\c{X})$ it is easily checked that these conditions coincide with (ii),(iii).
\begin{definition}\label{D2}
Let $\c{A}$ be a unital spatially continuous multiplication C*-algebra. An element $\rho\in\c{A}\check{\ot}\c{A}$ which satisfies (i)$'$,(ii)$''$,
(iii)$''$,(iv)$'$,(v)$'$ is called an algebraic (compatible) quantum metric on $\f{q}\c{A}$. In this case $(\c{A},\rho)$ is called an algebraic compact
quantum metric space.
\end{definition}
\begin{theorem}\label{T4}
Let $(\c{A}_1,\rho_1)$ and $(\c{A}_2,\rho_2)$ be algebraic compact quantum metric spaces.
Then $(\c{A}_1\oplus\c{A}_2,\rho)$ is an algebraic compact quantum metric space where
$\rho$ is as in Theorem \ref{T2}.
\end{theorem}
\begin{proof}
We only show that $\rho$ satisfies (iii)$''$. The other conditions are easily checked. Let $\c{A}:=\c{A}_1\oplus\c{A}_2$.
Let $m,m_1,m_2$ denote respectively the multiplications of $\c{A},\c{A}_1,\c{A}_2$ and let $\f{F},\f{F}_1,\f{F}_2$ denote the corresponding
fillips as in (iv)$'$. We have $\c{A}\check{\ot}\c{A}=\oplus_{i,j=1,2}\c{A}_i\check{\ot}\c{A}_j$. Let $\nu$ be a positive element of
$\c{A}\check{\ot}\c{A}$ with $m(\nu)=1$ and $\f{F}(\nu)=\nu$. Let $\nu_{ij}\in\c{A}_i\check{\ot}\c{A}_j$ be such that $\nu=\sum\nu_{ij}$.
Then $\nu_{ij}$ is positive and we have $m_i(\nu_{ii})=1_{\c{A}_1}$ and $\f{F}_i(\nu_{ii})=\nu_{ii}$. It follows that $\rho_i+\nu_{ii}$
is invertible in $\c{A}_i\check{\ot}\c{A}_i$, and $r1_{\c{A}_i}\ot1_{\c{A}_j}+\nu_{ij}$ is invertible in $\c{A}_i\check{\ot}\c{A}_j$ for $i\neq j$.
Thus $\rho+\nu$ is invertible in $\c{A}\check{\ot}\c{A}$.
\end{proof}
\begin{theorem}\label{T5}
Let $(\c{A}_1,\rho_1)$ and $(\c{A}_2,\rho_2)$ be algebraic compact quantum metric spaces such that $\c{A}_1$ is commutative.
Then $(\c{A}_1\check{\ot}\c{A}_2,\rho)$ is an algebraic compact quantum metric space where $\rho$ is as in Theorem \ref{T3}.
\end{theorem}
\begin{proof}
We only show that $\rho$ satisfies (iii)$'$. The other conditions are easily checked.
Let $m,m_2$ denote respectively the multiplications of $\c{A},\c{A}_2$ and let $\f{F},\f{F}_2$ denote the corresponding
fillips as in (iv)$'$. Let $\c{X}$ denote the Gelfand spectrum of $\c{A}_1$. Thus $\c{A}_1\cong\b{C}(\c{X})$ and
$\c{A}_1\check{\ot}\c{A}_2\cong\b{C}(\c{X},\c{A}_2)$, the algebra of $\c{A}_2$ valued continuous functions on $\c{X}$.
Let $\nu\in\b{C}(\c{X}\times\c{X},\c{A}_2\check{\ot}\c{A}_2)$ be a positive element with $m(\nu)=1$ and $\f{F}(\nu)=\nu$.
Then for every $x,y\in\c{X}$, $\nu(x,y)$ is positive,
$m_2(\nu(x,y))=1_{\c{A}_2}$, and $\f{F}_2(\nu(x,y))=\nu(x,y)$. Thus $\rho_2+\nu(x,y)$ is invertible in $\c{A}_2\check{\ot}\c{A}_2$.
It follows that $1_{\c{A}_1\check{\ot}\c{A}_1}\ot\rho_2+\nu$ (which is equal to the function  $(x,y)\mapsto\rho_2+\nu(x,y)$)
is invertible. Thus $\rho+\nu$ is also invertible.
\end{proof}
The main gap in our work is the lack of a nonclassical example:
\begin{problem}\label{P1}
Give an example of a nonclassical (algebraic) quantum metric.
\end{problem}
\bibliographystyle{amsplain}

\begin{thebibliography}{10}
\bibitem{Connes1}
A. Connes,
\emph{Noncommutative geometry},
Academic Press, 1994.
\bibitem{Connes2}
A. Connes,
\emph{On the spectral characterization of manifolds},
J. Noncommutative Geometry 7, no. 1 (2013): 1--82. (arXiv:0810.2088 [math.OA])
\bibitem{KadisonRingrose1}
R.V. Kadison, J.R. Ringrose,
\emph{Fundamentals of the theory of operator algebras: Advanced theory, Vol. 2},
American Mathematical Soc., 1997.
\bibitem{KuperbergWeaver1}
G. Kuperberg, N. Weaver,
\emph{A von Neumann algebra approach to quantum metrics/quantum relations},
Vol. 215, no. 1010. American Mathematical Society, 2012. (arXiv:1005.0353 [math.OA])
\bibitem{Rieffel3}
M.A. Rieffel,
\emph{Metrics on state spaces},
Doc. Math. 4 (1999): 559–-600. (arXiv:math/9906151 [math.OA])
\bibitem{Rieffel1}
M.A. Rieffel,
\emph{Group C*-algebras as compact quantum metric spaces},
Doc. Math 7 (2002): 605--651. (arXiv:math/0205195 [math.OA])
\bibitem{Rieffel2}
M.A. Rieffel,
\emph{Gromov-Hausdorff distance for quantum metric spaces/Matrix algebras converge to the sphere for quantum Gromov-Hausdorff distance},
Vol. 168, no. 796. American Mathematical Soc., 2004. (arXiv:math/0011063 [math.OA]) (arXiv:math/0108005 [math.OA])
\bibitem{Takesaki1}
M. Takesaki,
\emph{Theory of operator algebras I},
Reprint of the first (1979) edition, Encyclopaedia of Mathematical Sciences, 124, Operator Algebras and Noncommutative Geometry, 5. (2002).
\end{thebibliography}

\end{document}